\documentclass[10pt,reqno]{amsart}

\usepackage{amsmath}
\usepackage{amssymb}
\usepackage{amsthm}
\usepackage{enumerate}
\usepackage{amsbsy}
\usepackage{amsfonts}
\newtheorem{theo}{Theorem}[section]

\newtheorem{prop}[theo]{Proposition}

\newtheorem{lemm}[theo]{Lemma}
\newtheorem{rem}[theo]{Remark}

\newcommand{\al}{\alpha}
\newcommand{\be}{\beta}

\newcommand{\Ga}{\Gamma}

\newcommand{\Om}{\Omega}

\newcommand{\ep}{\epsilon }
\newcommand{\te}{\theta}
\newcommand{\De}{\Delta}

\newcommand{\pa}{\partial}

\newcommand{\R}{{\bf R}^n}

\newcommand{\ri}{\rightarrow}
\newcommand{\Rn}{{\bf R}^{n-1}}

\newcommand{\na}{\nabla}

\begin{document}
\baselineskip=18pt
\title[Fractional Laplacian]{Boundary integral operator for the fractional Laplace equation in  a bounded Lipschitz domain
}

\author{TongKeun Chang}
\address{Department of Mathematics, yonsei University, Seoul, South Korea}
\email{chang7357@yonsei.ac.kr}

\thanks{The author was supported by the
National Research Foundation of Korea NRF-2010-0016699}

\maketitle

\begin{abstract}
We study  the  boundary integral operator induced from fractional
Laplace equation in a bounded Lipschitz domain. As an application,
we study the boundary value problem of a fractional Laplace
equation.

\vspace*{.125in}


\vspace*{.125in}

\end{abstract}

\section{Introduction}
\label{sec1}
\setcounter{equation}{0}

In this paper we study a  boundary integral operator defined on the
boundary of a  bounded Lipschitz domain. Let $\Om$ be a bounded
Lipschitz domain and
\begin{align}\label{riesz kernel}
\Ga_{2s}(x) = c(n,s) \frac1{|x|^{n-2s}}
\end{align}
 is the Riesz kernel of order $2s, \,\, 0 < 2s < n$  in $\R$.
The layer potential of a fractional Laplacian for  $\phi \in L^2(\pa
\Om)$ is defined by
\begin{align}\label{boundary integral operator}
{\mathcal S}_s \phi (x) = \int_{\pa \Om}   \Ga_{2s} (x - Q) \phi(Q) dQ.
\end{align}
The boundary integral operator $S_s\phi = ({\mathcal S}_s \phi)|_{\pa \Om}$ is defined by the  restriction  of  ${\mathcal S}_s \phi$.

M. Z$\ddot{a}$hle\cite{Z} studied the Riesz potentials in  a general
metric space $(X, \rho)$ with Ahlfors $d$-regular measure $\mu$. He
showed $S_s: L^2(X, d\mu) \ri L^2_{2s}(X, d\mu), \,\, 0 < 2s < d< n$
is invertible,  where $L^2(X, d\mu)$ is decomposed by null space
$N(S_s) $ and orthogonal compliment of $N(S_s)$, that is, $L^2(X,
d\mu) = N(S_s)   \otimes L^2_{2s}(X, d \mu)$.

When $2s =2$, $\Ga_2$ is fundamental solution of Laplace equation in
$\R$ and \eqref{boundary integral operator} is single layer
potential of Laplace equation. The single layer potential  and
boundary layer potential of Laplace equation  were studied by many
mathematicians to show the solution of boundary value problem of
Laplace equation in a bounded domain (see \cite{FJR}, \cite{FMM},
\cite{JK} and  \cite{V}).

In this paper, we show the bijectivity of boundary layer potential
of a fractional Laplace equation in some  distribution space. Our
main result in this paper is stated as the  following.

\begin{theo}\label{theo1}
For $\frac12 < s <1$,   $S_s : H^{-s + \frac12}(\pa\Om) \ri
H^{s-\frac12} (\pa \Om)$ is bijective.
\end{theo}
Here,  the spaces $H^{-s + \frac12}(\pa\Om)$ and $ H^{s-\frac12} (\pa \Om)$ are defined in section \ref{sec2}.

The boundary integral operators (the single layer potential and the double layer potential) have been studied by
many mathematicians. The bijectivity of the  layer potentials has been used to show the existence of  solutions of partial differential equations in a
bounded Lipschitz domain or bounded Lipschitz cylinder (see
\cite{B1}, \cite{B2}, \cite{DKV}, \cite{FKV}, \cite{HL},  \cite{M}
and \cite{She}).

As in  many other literatures, we apply  the bijectivity of the
boundary  integral operator to the boundary value problem of
fractional Laplace equation in bounded  Lipschitz domain. The
fractional Laplacian of order $0 < 2s< 2$ of a function $v: \R \ri
{\bf R}$ is expressed by the formula
\begin{align*}
\De^s v (x) =  C(n,s) \int_{\R} \frac{v(x+y) - 2v (x) + v(x-y)}{|x
-y|^{n+2s}} dy,
\end{align*}
where $C(n,s)$ is some normalization constant.
The fractional Laplacian can also be defined as a pseudo-differential operator
\begin{align*}
\widehat{(-\De)^s  v}(\xi) = (2\pi  i| \xi|)^{2s} \hat{v} (\xi),
\end{align*}
where $\hat{ v}$ is the Fourier transform of $v$ in $\R$. In
particular, when $2s =2$  it is  naturally extended   to the Laplace
equation $\De v(x) = \sum_{1\leq i \leq n} \frac{\pa^2 }{\pa
x_i^2}v(x)$.

In this paper, we show that  the layer potential  defined by \eqref{boundary integral operator}  is in
$\dot{H}^s(\R) = \{ v \in L^2_{loc}(\R) \, | \,\| v \|_{\dot{H}^s (\R)}:
 = \int_{\R} |\xi|^{2s}|\hat v(\xi)|^2 d\xi < \infty \}$ (see theorem \ref{rem2})
and satisfies
$$
 \De^s v(x) = 0 \quad x \in \Rn \setminus \pa \Om
$$
(see \eqref{solution}).
Hence, from theorem \ref{theo1}, we obtain  the following theorem;
\begin{theo}\label{theo2-1}
Let $\frac12 < s <1$.  For given $f \in H^{s -\frac12}(\pa \Om)$,
the following equation
\begin{eqnarray}\label{main result}
\left\{\begin{array}{l} \De^s u =0 \quad \mbox{in} \quad \R \setminus \pa \Om,\\
u|_{\pa\Om} = f \in H^{s -\frac12}(\pa \Om),\\
\|u\|_{\dot{H}^s(\R)} \leq c \|f\|_{H^{s- \frac12}(\pa \Om)}.
\end{array}\right.
\end{eqnarray}
has a  unique solution. Furthermore, there exists  $\phi \in H^{-s
+\frac12}(\pa \Om)$ such that
\begin{align}\label{single layer solution}
u = {\mathcal S}_s \phi.
\end{align}
\end{theo}

The rest of the paper is organized as follows. In section
\ref{sec2}, we introduce several function spaces.

In section
\ref{sec3}, we will define the layer potential and the boundary
layer potential. The layer potential for $\phi \in H^{-\al}(\pa \Om)$ is defined by
\begin{align*}
{\mathcal S}_s \phi (x) = <\phi, \Ga_{2s} (x - \cdot)>,
\end{align*}
where $\Ga_s $ is Riesz kernel of order $2s$ defined in \eqref{riesz
kernel}  and $<\cdot, \cdot> $ is the duality paring between
$H^{\al} (\pa \Om) $ and $H^{-\al}(\pa \Om)$.  In particular, if
$\phi \in L^2(\pa \Om)$, then ${\mathcal S}_s \phi$ is defined by
\eqref{boundary integral operator}. The boundary layer potential
$S_s \phi$ for $\phi \in H^{-s +\frac12}(\pa \Om)$ is defined by
$S_s\phi = ({\mathcal S}_s \phi)|_{\pa \Om}$, where $F|_{\pa\Om}$ is
the restriction on $\pa \Om$ of  the function $F$ defined in $\R$.

In section \ref{sec3-1}, we study several properties of layer
potential.

In section \ref{sec5}, we   show the bijectivity of the boundary
layer potential $S_s : H^{-s\ +\frac12} (\pa \Om) \ri H^{s +
\frac12}(\pa \Om)$. 

The  probability is another  tool to represent the solution of a
fractional Laplace equation. Let $X_t$ be a $\al$-stable process on
$\R$ and $\tau_\Om = \inf\{ t> 0\, | \, X_t \notin \Om\}$. Note that
$X_t$ is discontinuous and $P\{\tau_\Om \in \pa \Om\} =0$  (see
\cite{B} and \cite{CS}). Hence, to represent the solution  in $\Om$
with probability of  a fractional Laplace equation,  we  need
information on $\Om^c$. Let  $\phi \in C^\infty(\Om^c)$ and define
function
\begin{align}\label{probability solution}
 u(x) = \left\{ \begin{array}{ll}
E_x \phi(X_{\tau_\Om}) & \quad x \in \Om, \\
\phi(x) & \quad x \in \Om^c,
\end{array}
\right.
\end{align}
where $E_x$ denote an expectation with respect to $P^x$  of the process
starting from $x \in \Om$. Then,  $u$ is a solution of
\begin{align*}
\left\{\begin{array}{ll}
\De^s  u =0 & \quad  \mbox{in} \quad \Om, \\
  u  = \phi  & \quad  \mbox{in} \quad \Om^c
\end{array}
\right.
\end{align*}
(see \cite{B} and \cite{CS}). Compared with our result, function
\eqref{probability solution} is a solution of a fractional Laplace
equation  in $\Om$ with information on $\Om^c$ and  function
\eqref{single layer solution} is a solution of factional Laplacian
in $\R \setminus \pa \Om$ with information on $\pa \Om$.

\section{Function spaces}
\setcounter{equation}{0}
\label{sec2}

In this paper, we consider   a   bounded Lipschitz
domain $\Om$ in $\R$.
The letters $x, y$ denote the  points in $\R$, and the
letters $P,Q$ denote the  points on the
 boundary $\pa \Om$ of the domain $\Om$. The letter $c$ denotes positive constant depending only on $n, \,\, s  $ and $\Om$.

For $ 0 < \al < 1$, we define the  Sobolev spaces $H^\al(\pa \Om)$ and $H^\al(\pa \Om)$ as
\begin{align*}
H^\al(\pa \Om) &=
\{\phi \in L^2(\pa \Om) \,\, | \,\, \int_{\pa \Om} \int_{\pa \Om}
\frac{|\phi(P) - \phi(Q)|^2}{|P - Q|^{n-1 + 2\al}} dPdQ < \infty \},\\
H^\al( \Om) &=
\{\phi \in L^2(\Om) \,\, | \,\, \int_{ \Om} \int_{ \Om}
\frac{|\phi(x) - \phi(y)|^2}{|x-y|^{n + 2\al}} dydx < \infty \}.
\end{align*}
We define  $H^{-\al} (\pa \Om) = (H^\al (\pa \Om))^*$ by
   dual spaces of  $H^\al(\pa \Om)$.  For $ 0 <\al < 1$, Sobolev spaces  $H^\al(\R ), \,\, H^\al(\Om)$ are defined in a similar manner.

 Now, we define homogeneous Sobolev space $\dot{H}^{\al}(\R)$.
For $\al \in {\bf R}$, the homogeneous Sobolev space  $\dot{H}^\al
(\R)$ is set of distributions   satisfying
\begin{equation}\label{homogeneous norm}
\| v\|^2_{\dot{H}^\al(\R)} : =\int_{{\bf R}^n}|\xi|^{2\al}|\hat{ v}(\xi)|^2 d\xi,
\end{equation}
where $\hat{v}$ means Fourier transform of $v$ in $\R$.
\begin{rem}
Let $0 < \al < 1$.
\begin{itemize}
\item[(1)]
By simple calculation,   for $v \in \dot{H}^\al(\R)$,  we obtain
\begin{align} \label{homogeneous norm2}
\int_{{\bf R}^n} \int_{{\bf R}^n}
\frac{|v(x )-   v(y)|^2}{|x-y|^{n + 2\al}} dydx =C\int_{{\bf R}^n}|\xi|^{2\al}|\hat{v}(\xi)|^2 d\xi.
\end{align}
\item[(2)]
For $v \in \dot{H}^\al  (\R)$, we get $\De^s v \in \dot{H}^{ \al
-2s} (\R)  $. In particular, if $ 0 < \al<  2s$, then, $\De^s v  \in
\dot{H}^{\al - 2s}(\R)   =  ( \dot{H}^{\al -2s}(\R))^*$ is defined
by
\begin{align}
< \De^s v, \psi> = \int_{\R} (2\pi i |\xi|)^{2s} \hat{v}(\xi) \bar {\hat{\psi}} (\xi) d \xi
\end{align}
for $\psi \in \dot{H}^{2s -\al}(\R)$.

\end{itemize}

\end{rem}

\section{Layer potential and boundary layer potential}
\setcounter{equation}{0}
\label{sec3}

Given $\phi \in  H^{-\al} (\pa \Om), \,\, \al \geq 0$,  we define layer potential by
\begin{eqnarray}\label{single1}
{\mathcal S}_s  \phi(x) =  < \phi, \Ga_{2s}(x - \cdot)>
                  \quad  x \in \R \setminus \pa \Om,
\end{eqnarray}
where $<\cdot, \cdot> $ is the duality paring between $H^{-\al}(\pa
\Om)$ and $H^{\al}(\pa \Om)$. In particular, if $\phi \in L^2(\pa
\Om)$, then
\begin{align}\label{definition}
{\mathcal S}_s \phi (x) = \int_{\pa \Om} \Ga_{2s} (x -Q) \phi(Q) dQ, \quad x \in  \R \setminus \pa \Om.
\end{align}
Note that ${\mathcal S}_s \phi$ is in $C^\infty(\R \setminus \pa
\Om)$. For $0 < \al $  and for  large $|x|$, we have
\begin{align}\label{behavior2}
|\na^\be {\mathcal S}_s(x)| \leq \|\phi\|_{H^{-\al} (\pa \Om)} \|\na^\be \Ga_{2s} (x -
\cdot)\|_{H^\al(\pa \Om)} \leq c \|\phi\|_{H^{-\al} (\pa \Om)}
\frac{1}{|x|^{n-2s+ |\be|}},
\end{align}
where $\be = (\be_1, \be_2, \cdots, \be_n)  \in   ( {\bf N} \cup \{ 0 \} )^n$ with $|\be| = \sum_{k} \be_k$.

We will use the following proposition late on (see \cite{JW}).
\begin{prop} \label{trace}
For $0 < s < 1$, the operator
${\mathcal R} : H^{s + \frac12}(\R) \ri H^{s } (\pa \Om)$ defined by $
{\mathcal R} (F) = F|_{\pa \Om}$ is bounded. That is, there is
constant $c> 0$ such that
\begin{eqnarray*}
\|{\mathcal R} (F)\|_{H^{s} (\pa \Om)} \leq c\|F\|_{H^{s +\frac12}
(\R)},
\end{eqnarray*}
where $c > 0$ depend only on  $n, \,\, s $ and $\Omega$ .
\end{prop}

\begin{rem}\label{rem}
Let $U$ be a  Lipschitz bounded open subset of $\R$ such that $\pa \Om
\subset U$. For $F \in H^{s +\frac12} (U)$, let us $\tilde F \in H^{s + \frac12}(\R)$ be a
Stein's extension (see proposition 2.4  in \cite{JK}).
That is, $\tilde F|_U = F$ and $\|\tilde F\|_{H^{s +\frac12}(\R)} \leq c \| F\|_{H^{s +\frac12} (U)}$,
where $c$ depends only on  $s$ and $ U$.  Then, by the
proposition \ref{trace}, we get $\| F|_{\pa  \Om} \|_{H^s (\pa \Om)} \leq
c \|\tilde F\|_{H^{s +\frac12}(\R)} \leq c \| F\|_{H^{s +\frac12} (U)}$.
\end{rem}

We introduce a Riesz potential $I_s$, $0  < s < n$, by
$$
I_s \psi(x) = c(n,s) \int_{\R} \frac{1}{|x -y|^{n-s}}
\psi(y) dy\mbox{ for }\psi
\in C^\infty_c (\R).
$$
The following proposition is well known fact and will be useful in the subsequent estimates (see chapter 5 of \cite{St}).
\begin{prop}\label{prop2}
\begin{itemize}

\item[1).] Let $1 \leq p < q < \infty, \,\, \frac1q = \frac1p -
\frac{s}{n}$. Then
$$
I_s:L^p (\R) \ri L^q(\R)
$$
is bounded.

\item[2).] The Fourier transform of $c(n, s)|x|^{-n + s}$ is  $
|\xi|^{-s}$, in the sense that
$$
\int_{\R} c(n,s)|x|^{-n + s} \psi (x) dx = \int_{\R}
|\xi|^{-s } \overline{\hat \psi(\xi)} d\xi.
$$
Hence,  for $\psi \in C_c^\infty(\R) $, $\hat{I_s \psi  } (\xi) = |\xi|^{-s} \hat{ \psi}(\xi )   $ for $\xi \in \R$.

\item[3).] Let $ 0< s_1, \, s_2< n  \quad
s_1 + s_2 < n.$ Then
$$
c(n,s_1) c(n,s_2) \int_{\R} \frac{1}{|x-y|^{n- s_1}}
\frac{1}{|y|^{n- s_1}} dy = c(n, s_1 + s_2) \frac{1}{|x|^{n
-s_1 -s_2}}.
$$
\end{itemize}
\end{prop}

\begin{lemm}\label{lemma2}
Let $B_R$ be the open ball in $\R$ centered at the  origin with radius $R$.
Then, the integral operator $I^R_{2s}: C^\infty_c (B_R) \ri H^{\al}(B_R)$ defined by
\begin{align*}
I^R_{2s} F(x)
=  \int_{B_R} \Ga_{2s}(x-y)  F(y) dy
\end{align*}
is continuously  extended to $I^R_{2s}:H^{-\al}_0(B_R)\ri H^{-\al + 2s }(B_R), \,\, 0 \leq \al \leq 2s$, where
$H^{-\al}_0(B_R) = (H^\al(B_R))^*$.
\end{lemm}

\begin{proof}
By  2) of proposition \ref{prop2}, we have $I_{2s}: L^2(\R) \ri \dot{H}^{2s}(\R)$ is continuous.
By 1) of proposition \ref{prop2}, this implies that $I^R_{2s} : L^2(B_R) \ri H^{2s}(B_R)$ is a bounded operator. 
Then, for $\phi \in C^\infty_c(B_R), \,\, F \in L^2(B_R)$, we get
\begin{align*}
\int_{B_R} F(x) I^R_{2s} \phi(x) dx & = \int_{B_R} I^R_{2s} F(x) \phi(x)dx\\
 & \leq \| \phi\|_{H^{-2s}_0(B_R)} \| I^R_{2s} F\|_{H^{2s}(B_R)}\\
 & \leq c\| \phi\|_{H^{-2s}_0(B_R)} \|  F\|_{L^2(B_R)}.
\end{align*}
This implies $ \|I^R_{2s} \phi\|_{L^2(B_R)} \leq c \| \phi\|_{H^{-2s}_0(B_R)}$.
Since $C^\infty_c(B_R) $ is a dense subset of $H^{-2s}_0(B_R)$ (see remark 2.7 in \cite{JK}),
we get lemma \ref{lemma2} for $\al = 2s$.
Note that $H^{-2s(1-\te)}_0(B_R)$ and $H^{ 2s\te}(B_R)$ are real interpolation spaces with
$(H^{-2s}_0(B_R), L^2(B_R))_\te = H_0^{-2s (1-\te)}(B_R)$ and
$(L^2(B_R), H^{2s}(B_R))_\te = H^{ 2s \te}(B_R), \,\, 0 < \te < 1$ (see proposition 2.4 and remark 2.7 in  \cite{JK}).
Taking $\te =\frac{2s -\al}{2s}$, we get lemma \ref{lemma2} for $0 \leq \te\leq 1$.
\end{proof}

\begin{lemm}\label{lemm3}
 Let $\phi \in L^2(\pa \Om)$ and  $u  = {\mathcal S}_s \phi$  be defined by \eqref{definition}.
 Fix $\ep > 0$ and define $\Om^\ep = \{ x \in \R \, | \, d(x,\pa \Om) < \ep \}$.  Then for $ F \in C^\infty_c (B_R) $ we have
\begin{eqnarray}\label{l2}
\int_{B_R \setminus \Om^\ep} F(x) u(x) dx= \int_{\pa \Om} \phi(Q) \int_{B_R \setminus \Om^\ep}F(x) \Ga_{2s} (x- Q) dx  dQ,
\end{eqnarray}
where $d(x, \pa \Om)$ is distance between $x$ and $\pa \Om$, and $B_R$ is the open ball  centered at the  origin with radius
$R$ such that $\Om \subset B_{\frac12 R}$.
\end{lemm}

\begin{proof}
Note that by H$\ddot{o}$lder inequality, we get
\begin{align*}
\int_{B_R \setminus \Om^\ep} \int_{\pa \Om}|F(x)| \Ga_{2s}(x -Q) |\phi(Q)| dQ dx
\leq c  R^{\frac{n}2}|\pa \Om|^\frac12  \ep^{-n +2s} \| F\|_{L^2(B_R \setminus \Om^\ep)  } \| \phi \|_{L^2(\pa \Om)  },
\end{align*}
where $| \cdot |$ means the Lebesgue measure.
Hence, by Fubini's theorem, we get
\begin{eqnarray*}
\int_{B_R \setminus \Om^\ep} F(x) u(x) dx = \int_{B_R \setminus \Om^\ep} F(x) \int_{\pa \Om} \Ga_{2s}(x-Q) \phi(Q) dQ dx\\
=\int_{\pa \Om} \phi(Q) \int_{B_R \setminus \Om^\ep}F(x) \Ga_{2s} (x- Q) dx  dQ.
\end{eqnarray*}

\end{proof}

\begin{theo}\label{local}
Let $\frac12 < s <1$ and $\phi \in H^{-s +\frac12}(\pa \Om)$.
Set $ u  = {\mathcal S}_s \phi $ defined in \eqref{single1}.  Then, for all $B_R$ satisfying $\Om \subset B_{\frac12 R}$,
\begin{align}\label{theo2}
\| u\|_{H^{s } (B_R)}
\leq c_R \|\phi\|_{H^{-s +\frac12}(\pa \Om)}.
\end{align}
Moreover, for all $F \in C^\infty_c(B(R) )$
\begin{align}\label{duality}
\int_{B(R) } F(x) u(x) dx = <\phi, \int_{B(R)  }F(x) \Ga_{2s} (x- \cdot)
dx>.
\end{align}

\end{theo}

 \begin{proof}
 Assume $\phi \in L^2(\pa \Om)$. Let $F \in C^\infty_c(B_R)$.
From   lemma \ref{lemma2}, we get
\begin{align*}
&\| \int_{B_R \setminus \Om_\ep} F(x) \Ga_{2s} (x -\cdot)dx - \int_{B_R } F(x) \Ga_{2s}(x -\cdot)dx \|_{H^{2s}(B_R)} \\
&= \| \int_{  \Om^\ep}F(x) \Ga_{2s} (x- \cdot) dx\|_{H^{2s}(B_R)} \leq c_R \|    F \|_{L^2 (\Om^\ep)} \ri 0
\quad \mbox{as} \quad \ep \ri 0,
\end{align*}
where  $\Om^\ep$ is defined in lemma \ref{lemm3} and $c_R$ is independent of $\ep$.  Applying remark \ref{rem},  this implies
\begin{align*}
\int_{B_R \setminus \Om_\ep} F(x) \Ga_{2s}(x -\cdot)dx \ri \int_{B_R } F(x) \Ga_{2s}(x -\cdot)dx
  \quad \mbox{in} \quad H^{2s -\frac12}(\pa \Om).
\end{align*}
In particular, $\int_{B_R \setminus \Om_\ep} F(x) \Ga_{2s}(x -\cdot)dx \ri \int_{B_R } F(x) \Ga_{2s}(x -\cdot)dx$ in $L^2(\pa \Om)$.
Hence, sending $\ep$ to the zero in \eqref{l2}, we obtain
\begin{align}\label{L2}
\int_{B_R} F(x) u(x) dx = \int_{\pa \Om} \phi(Q) \int_{B_R }F(x) \Ga_{2s} (x- Q) dx  dQ .
\end{align}
Since $L^2(\pa \Om)$ is dense subspace of $H^{-s +\frac12}(\pa \Om)$  and
$ \int_{B_R }F(x) \Ga_{2s} (x- \cdot) dx   \in H^{s -\frac12}(\pa \Om)$ by lemma \ref{lemma2}
and proposition  \ref{trace}, we obtain  \eqref{duality} from \eqref{L2}.
We again apply  proposition  \ref{trace},  lemma \ref{lemma2} and \eqref{duality}  to obtain the following
\begin{align}\label{result1}
|\int_{B_R } F(x) u(x) dx| & \leq c_R \|\phi
\|_{H^{-s+\frac12}(\pa \Om)} \|\int_{B_R }   F(x) \Ga_{2s}(x- \cdot)dx\|_{H^{s-\frac12} (\pa \Om)} \nonumber\\
 &\leq c_R \|\phi
\|_{H^{-s+\frac12}(\pa \Om)} \|\int_{B_R }   F(x) \Ga_{2s}(x- \cdot)dx\|_{H^{s} (B_R)}  \nonumber\\
& \leq c_R \|\phi
\|_{H^{-s+\frac12}(\pa \Om)}\| F\|_{H^{-s}_0({B_R} )}.
\end{align}
Since $C^\infty_c(B_R)$ is  dense in
$H^{-s}_0(B_R) = ( H^s(B_R))^* $ and $L^2(\pa \Om)$ is dense in $H^{-s+\frac12}(\pa \Om)$,
 \eqref{result1} holds for all $F \in H^{-s}_0(B_R)$ and $\phi \in H^{-s+\frac12}  (\pa \Om)$.
 Since  $H^s(B_R)$ is reflexive, we obtain \eqref{theo2}.
\end{proof}

By the proposition \ref{trace}, remark \ref{rem} and theorem \ref{local}, we have
the following theorem.
\begin{theo}\label{coro}
For $\frac12 < s <1$,
$$
S_s : H^{-s +\frac12}(\pa \Om) \ri H^{s -\frac12}(\pa \Om), \quad
S_s \phi = ({\mathcal S}_s \phi)|_{\pa \Om}
$$
is bounded operator, where $({\mathcal S}_s \phi )|_{\pa \Om}$ is
restriction of ${\mathcal S}_s \phi $ over $\pa \Om$.
\end{theo}

\section{Properties of layer potential}\label{sec3-1}
\setcounter{equation}{0}

\begin{theo}\label{rem2}
Let  $ \frac12 < s < 1$ and $\phi \in H^{-s + \frac12}(\pa \Om)$ and
$u = {\mathcal S}_s \phi$ be a layer potential defined in
\eqref{single1}. Then $u \in \dot{H}^s (\R)$ and
\begin{align}\label{homo-norm}
\int_{\R}\int_{\R} \frac{|u(x) -   u( y)|^2}{| x-y|^{n+2s}} dydx
\leq c \| \phi \|^2_{H^{-s
+\frac12}(\pa \Om)}.
\end{align}
\end{theo}
\begin{proof}
Let $B_R$ be an open ball whose center is origin and radius is $R$ such that $\Om \subset B_{\frac12 R}$.
We divide the left-hand side of \eqref{homo-norm} with three parts
\begin{align}\label{homo-divide}
\begin{array}{ll}\vspace{2mm}
 &A_1 = \int_{|x| \leq R} \int_{|y| \leq R}  \frac{|u(x ) -  u( y)|^2}{| x- y|^{n+2s}} dydx, \quad
 A_2 =  2\int_{|x| \leq R} \int_{|y| \geq  R} \frac{|u(x ) -  u( y)|^2}{| x- y|^{n+2s}} dydx, \\
& \hspace{35mm} A_3 = \int_{|x| \geq R} \int_{|y| \geq R} \frac{|u(x ) -  u( y)|^2}{| x- y|^{n+2s}} dydx.
\end{array}
\end{align}
By  the theorem \ref{local}, $A_1$  is dominated by $ \| \phi \|^2_{H^{-s
+\frac12}(\pa \Om)}$.  For $|x| \leq R$ and $|y| \geq 2R$, we get that $|x-y| \geq |y| -|x| \geq |y| -R\geq \frac12 |y|  $.
Note that  by \eqref{behavior2}, for $|y| \geq 2R$,  we have  that $|u(y)|^2 \leq c  |y|^{-2n + 4s}\| \phi \|^2_{H^{-s
+\frac12}(\pa \Om)}$.  Hence, by  the theorem \ref{local}, we have
\begin{align*}
A_2 & \leq   2\int_{|x| \leq R} \int_{ R \leq |y| \leq  2R} \frac{|u(x ) - u( y)|^2}{| x- y|^{n+2s}} dydx
+ 8\int_{|x| \leq R} \int_{|y| \geq 2R} \frac{|u(x ) |^2 +  | u( y)|^2}{| y|^{n+2s}} dydx \\
& \leq c_R \| u\|^2_{H^s(B(2R)) }   + c \| \phi \|^2_{H^{-s
+\frac12}(\pa \Om)}  \int_{|x| \leq R} \int_{|y| \geq 2R} \frac{  1}{|  y|^{3n-2s}} dydx\\
& \leq c_R \| \phi\|^2_{H^{-s + \frac12}(\pa \Om) }.
\end{align*}

We divide $A_3$ with two parts;
\begin{align}\label{homo-divide2}
\int_{|x| \geq R} \int_{|y| \geq R, |x-y| \leq \frac12 |x|}  \frac{|u(x ) -  u( y)|^2}{| x- y|^{n+2s}} dydx
+ \int_{|x| \geq R} \int_{|y| \geq R, |x-y| \geq \frac12 |x|}  \frac{|u(x ) -  u( y)|^2}{| x- y|^{n+2s}} dydx.
\end{align}
For $|x| \geq R, \,\, |x-y| \leq \frac12 |x|$, applying mean-value theorem,  there is a $\eta $ between $x$ and $y$
such that $u(x) - u(y) = \na u(\eta) \cdot (x-y)$. Note that $|x-\eta| \leq \frac12 |x|$  and hence $|\eta| \geq \frac12 |x| \geq \frac12 R$.
Hence, by \eqref{behavior2},  the first term of \eqref{homo-divide2} is dominated by
\begin{align*}
  &\int_{|x| \geq R} \int_{ |x-y| \leq \frac12 |x|}  \frac{|\na u(\eta )|^2 }{| x- y|^{n+2s-2}} dydx \\
& \leq  c\|\phi \|^2_{H^{-s +\frac12}(\pa \Om)} \int_{|x| \geq R} \frac{1}{|x|^{2n -4s +2}}
\int_{ |x-y| \leq \frac12 |x|}  \frac{1}{| x- y|^{n+2s-2}} dydx \\
& \leq c \|\phi \|^2_{H^{-s +\frac12}(\pa \Om)} \int_{|x| \geq R} \frac{1}{|x|^{2n -2s}} dx\\
& = cR^{-n+2s}  \|\phi \|^2_{H^{-s +\frac12}(\pa \Om)}.
\end{align*}
Since $|x|, \,\, |y| \geq R$, by \eqref{behavior2}, the second term
of \eqref{homo-divide2} is dominated by
\begin{align}\label{homo-divide3}
\begin{array}{ll}\vspace{2mm}
&  \int_{|x| \geq R} \int_{|y| \geq R, |x-y| \geq \frac12 |x|}  \frac{|u(x )|^2 + | u( y)|^2}{| x- y|^{n+2s}} dydx\\ \vspace{2mm}
& \leq  \|\phi \|^2_{H^{-s +\frac12}(\pa \Om)} \int_{|x| \geq R}   \frac{1}{|x|^{2n -4s}}
\int_{|y| \geq R, |x-y| \geq \frac12 |x|} \frac{1}{|x-y|^{n+ 2s}} dydx\\
 & +  \|\phi \|^2_{H^{-s +\frac12}(\pa \Om)}
\int_{|x| \geq R}  \int_{|y| \geq R, |x-y| \geq \frac12 |x|}    \frac{1}{|x-y|^{n+ 2s}} \frac{1}{|y|^{2n -4s}}dydx.
\end{array}
\end{align}
The first term  of right-hand side of \eqref{homo-divide3} is dominated by $R^{-n + 2s} \|\phi \|^2_{H^{-s +\frac12}(\pa \Om)}$.
Note that
\begin{align*}
\int_{|x| \geq R}  \int_{ R \leq |y| \leq 2 |x|}    \frac{1}{|x|^{n+ 2s}} \frac{1}{|y|^{2n -4s}}dydx
& \leq  c  \left\{ \begin{array}{l}
 R^{-n +4s} \int_{|x| \geq R}    \frac{1}{|x|^{n+ 2s}} dx, \quad 2n -4s > n,\\
\int_{|x| \geq R}    \frac{\ln |x|}{|x|^{n+ 2s}}  dx, \quad 2n -4s =n,\\
\int_{|x| \geq R}    \frac{1}{|x|^{2n- 2s}} dx, \quad 2n -4s <n
\end{array}
\right.\\
& \leq c R^{-n + 2s}  \ln R.
\end{align*}
Then, the second  term  of right-hand side of \eqref{homo-divide3} is dominated by
\begin{align*}
&   \|\phi \|^2_{H^{-s +\frac12}(\pa \Om)} \int_{|x| \geq R}  \int_{|y| \geq R, |x-y| \geq \frac12 |x|}    \frac{1}{|x-y|^{n+ 2s}} \frac{1}{|y|^{2n -4s}}dydx \\
&  \leq c  \|\phi \|^2_{H^{-s +\frac12}(\pa \Om)}
           \Big( \int_{|x| \geq R}  \int_{ R \leq |y| \leq  2|x|}    \frac{1}{|x|^{n+ 2s}} \frac{1}{|y|^{2n -4s}}dydx
+ \int_{|x| \geq R}  \int_{|y| \geq 2 |x| }     \frac{1}{|y|^{3n -2s}}dydx \Big) \\
&  \leq c R^{-n + 2s}  \ln R  \|\phi \|^2_{H^{-s +\frac12}(\pa \Om)}.
\end{align*}
Therefore, we  showed  that $ A_1 + A_2 + A_3  \leq c_R \|\phi \|^2_{H^{-s +\frac12}(\pa \Om)}$ and hence showed  \eqref{homo-norm}.
\end{proof}

\begin{theo}
Let  $ \frac12 < s < 1$ and $\phi \in H^{-s + \frac12}(\pa \Om)$ and
$u = {\mathcal S}_s \phi$ be a layer potential defined in
\eqref{single1}. Then,
\begin{align}\label{Fourier}
\hat{u}(\xi) = |\xi|^{-2s} < \phi,  e^{2\pi i \xi \cdot  }>.
\end{align}
\end{theo}
\begin{proof}
Let  $\phi \in L^2(\pa \Om)$ and $\psi \in C^\infty_c (\R)$. By  \eqref{L2} and 2)
of proposition \ref{prop2}, we have
\begin{align*}
\int_{\R} u (x) \psi(x) dx  &= c(n,s)\int_{\pa \Om} \phi (Q)
\int_{\R}
\frac{1}{|x-Q|^{n-2s}} \psi(x) dx dQ\\
 & =   \int_{\pa \Om} \phi (Q) \int_{\R} |\xi|^{-2s} e^{2\pi i \xi \cdot Q}
 \overline{\hat{\psi}(\xi)}
 d\xi  dQ\\
 & =  \int_{\R} \overline{\hat{\psi}(\xi)} |\xi|^{-2s} \int_{\pa \Om} \phi(Q) e^{2\pi
 i \xi \cdot Q} dQd\xi.
\end{align*}
Hence, we get
\begin{align*}
\hat{u}(\xi) = |\xi|^{-2s} \int_{\pa \Om} \phi(Q)  e^{2\pi i
\xi \cdot Q } dQ.
\end{align*}
Since $L^2(\pa \Om)$ is dense in $H^{-s +\frac12}(\pa \Om)$, we get \eqref{Fourier} for all  $\phi \in H^{-s + \frac12}(\pa \Om)$.
\end{proof}

\begin{theo}
Let  $ \frac12 < s < 1$ and $\phi \in H^{-s + \frac12}(\pa \Om)$ and
$u = {\mathcal S}_s \phi$ be a layer potential defined in
\eqref{single1}. Then,
\begin{align}\label{solution}
\De^s u =0, \quad \mbox{in} \quad \R \setminus \pa \Om.
\end{align}

\end{theo}
\begin{proof}
Suppose that $\phi \in L^2 (\pa \Om)$ and $\psi \in C^\infty_c
(\R \setminus \pa \Om)$, then
\begin{align*}
<\De^s u, \psi>  & = -\int_{\R}
 |\xi|^{2s} \hat {u}(\xi) \overline{\hat{\psi}(\xi)} d\xi \\
 &= - \int_{\R} \overline{\hat{\psi} (\xi)} 
     \int_{\pa \Om} e^{-2\pi i \xi \cdot Q}  \phi(Q) dQ
 d\xi\\
 & = -  \int_{\pa \Om} \phi (Q)  \overline{\int_{\R} e^{2\pi i \xi \cdot Q }
 \hat{\psi} (\xi) d\xi} dQ\\
 & = -\int_{\pa\Om} \phi(Q) \psi (Q) dQ\\
 & =0.
\end{align*}
Since $L^2 (\pa \Om)$ is dense subspace of $H^{-s + \frac12}(\pa
\Om)$, we get  \eqref{solution}for all $\phi \in H^{-s + \frac12 }(\pa \Om)$.
\end{proof}

\begin{theo}
Let  $ \frac12 < s < 1$ and $\phi \in H^{-s + \frac12}(\pa \Om)$ and
$u = {\mathcal S}_s \phi$ be a layer potential defined in
\eqref{single1}. Then,
\begin{align}\label{paring}
<\phi, u> =- \int_{\R} | \xi|^{2s} |\hat{u}|^2.
\end{align}

\end{theo}
\begin{proof}
Let $\phi \in L^2(\pa \Om)$. Then, it is well-known that the outer
and inner normal derivatives $\frac{\pa {\mathcal S}_1 \phi}{\pa
{\bf n}^+}, \,\, \frac{\pa {\mathcal S}_1 \phi}{\pa {\bf n}^-}$ of
${\mathcal S}_1 \phi$  are in $L^2 (\pa \Om)$ and  $-\phi =
\frac{\pa {\mathcal S}_1 \phi}{\pa {\bf n}^+} + \frac{\pa
{\mathcal S}_1 \phi}{\pa {\bf n}^-}$ (see \cite{V}). Hence, we have
\begin{align*}
-\int_{\pa \Om} u(Q) \phi(Q) dQ & = \int_{\pa \Om} u(Q) (\frac{\pa
{\mathcal S}_1 \phi}{\pa {\bf n}^+} + \frac{\pa {\mathcal S}_1
\phi}{\pa
{\bf n}^-})  \\
& = \int_{\R} \na u \cdot \na {\mathcal S}_1 \phi  \\
& =\int_{\R} |\xi|^2 |\xi|^{-2} \hat{u} \overline{\int_{\pa
\Om} e^{-2\pi i \xi
\cdot Q} \phi(Q) dQ} d\xi\\
& = \int_{\R} |\xi|^{2s} |\xi|^{-2s} \hat{u}
\overline{\int_{\pa \Om} e^{-2\pi i \xi \cdot Q} \phi(Q) dQ} d\xi  \\
& = \int_{\R} |\xi|^{2s} |\hat{u}|^2.
\end{align*}

Since $L^2(\pa \Om)$ is a dense subset of $H^{-s +\frac12} (\pa
\Om)$, we get \eqref{paring} for all $\phi \in H^{-s + \frac12} (\pa\Om)$.
\end{proof}

\section{Proof of  Theorem \ref{theo1}}\label{sec5}
\setcounter{equation}{0}

 \begin{lemm}\label{lemma1}

 Let $\frac{1}{2}<s<1$. Then
$S_s : H^{-s + \frac12} (\pa \Om) \ri H^{s -\frac12} (\pa \Om)$ is
one-to-one.
 \end{lemm}
\begin{proof}
Suppose that $S_s \phi =0$ for some $\phi \in H^{-s +\frac12} (\pa
\Om)$. Then, by \eqref{paring} and decay of $u$ at infinity, we get $u \equiv 0$ in $\R$. By 3)
of Proposition \ref{prop2}, we have
\begin{align*}
0=I_{2 -2s} u(x)& = c_n\int_{\R} \frac{1}{|x -y|^{n -(2-2s)}}
< \phi , \Ga_{2s}(y - \cdot)> dy\\
 &= < \phi,  \int_{\R}  \frac{1}{|x -y|^{n-(2-2s)}}
\frac{1}{| y - \cdot|^{n-2s}} dy \\
 & =< \phi,\Ga_1(x -\cdot)>.
\end{align*}
Hence, $S_1 \phi =0$. Note that $S_1 : H^{-\al} (\pa \Om) \ri H^{-\al
+ 1} (\pa \Om)$ is bijective  for $0\leq \al \leq 1$ (see \cite{FMM} and
\cite{JK}). Since $ 0 < s -\frac12  < \frac12 $,  we get $\phi =0$ and so
$S_s : H^{-s + \frac12} (\pa \Om) \ri H^{s -\frac12} (\pa \Om)$ is
one-to-one.
\end{proof}

\begin{lemm}\label{lemm}
Let $\frac{1}{2}<s<1$. Then
$S_s : H^{-s + \frac12} (\pa \Om) \ri H^{s -\frac12} (\pa \Om)$
has a closed range.
\end{lemm}
\begin{proof}
Suppose that $S_s \phi_k \ri f$ in $H^{s -\frac12}(\pa \Om)$
for some sequence $\{\phi_k\}$ in $H^{-s+\frac12}(\pa \Om)$. If
$\{\phi_k \}$ is bounded in $H^{-s+\frac12}(\pa \Om)$, then it is
done since there are a subsequence (we say
 $\{\phi_k\}$) and $\phi \in H^{-s+ \frac12}(\pa \Om)$ such that $\phi_k$ weakly converges to $\phi$ and we
 observe
\begin{eqnarray*}
<\psi, f> = \lim_{k \ri \infty}< \psi, S_s\phi_k> = \lim_{k \ri
\infty} \ll S_s^* \psi, \phi_k\gg = \ll S_s^* \psi, \phi\gg = < \psi,
S_s \phi>
\end{eqnarray*}
for all $\psi \in H^{-s+ \frac12} (\pa \Om)$, where $S_s^*:
H^{-s+\frac12}(\pa \Om) \ri (H^{-s +\frac12} (\pa \Om))^*$ is a dual
operator of $S_s$ and $\ll \cdot, \cdot \gg$ is the duality paring between
$(H^{-s+ \frac12}(\pa \Om))^*$ and $H^{-s+\frac12}(\pa \Om)$.
Hence we conclude that $S_s\phi=f$.

Now we would like to show that  $\{\phi_k \}$ cannot be unbounded.
Suppose that $\{\phi_k \}$ is unbounded in $H^{-s +\frac12} (\pa
\Om)$. Let $\Phi_k = \frac{\phi_k}{\|\phi_k\|_{H^{-s +
\frac12}(\pa \Om)}}$. Then $\|\Phi_k\|_{H^{-s + \frac12}(\pa \Om)}
=1$ for each $k$. Since $\{ \Phi_k \}$ is bounded in $H^{-s
+\frac12}(\pa \Om)$, there is a subsequence(we again say
$\{\Phi_k \}$) of $\{\Phi_k \}$  and $ \Phi \in H^{-s + \frac12
}(\pa \Om)$ such that $\{ \Phi_k \}$  converges weakly to $\Phi$ in
$H^{-s +\frac12}(\pa \Om)$. Since $S_s \Phi_k  \ri 0$ in $H^{s
-\frac12}(\pa \Om)$ and $S_s$ is one-to-one by Lemma \ref{lemma1},
we conclude that $\Phi =0$. Note that ${\mathcal S}_1 : H^{-s + \frac12} (\pa \Om) \ri H^{-s
+2 } (\Om)$ is bounded, $\frac12 < s < 1$, and $H^{-s+2 } (\Om)$
is compactly imbedded into $H^1(\Om)$ (see \cite{FMM} and
\cite{JK}). Hence $\Phi_k\ri 0$ weakly in $H^{-s+\frac{1}{2}}(\pa
\Om)$, $\frac{1}{2}<s<1$, implies
\begin{align}
\label{m2}{\mathcal S}_1 \Phi_k \ri 0\mbox{ in }H^1 (\Om).
\end{align}
Let $u_k = S_s \Phi_k$. Then, by \eqref{paring}, we get $\int_{\R}
|\xi|^{2s} |\hat{u}_k (\xi)|^2 d\xi=<\Phi_k,u_k> \ri 0$. Note that
\begin{align}\label{m1}\nonumber
&\int_{\R} \int_{\R} \frac{|{\mathcal S}_1
\Phi_k(x+y) -2{\mathcal S}_1 \Phi_k (x) + {\mathcal S}_1 \Phi_k
(x-y)|^2}{|y|^{n+2(2 -s)}}dydx\\
&=\int_{\R} |\xi|^{4
- 2s} |\widehat{{\mathcal S}_1 \Phi_k} (\xi)|^2 d\xi =\int_{\R} |\xi|^{2s} |\hat{u}_k (\xi)|^2 d\xi   \ri 0.
\end{align}

Combining \eqref{m1} and \eqref{m2}, we have ${\mathcal S}_1
\Phi_k \ri  0$ in $H^{-s +2 } (\Om)$.
This implies that $S_1 \Phi_k \ri 0 $ in $H^{-s +\frac32} (\pa
\Om)$. Note that $S_1 : H^{-s + \frac12} (\pa \Om) \ri H^{-s +
\frac32 } (\pa \Om)$ is bijective, $\frac12<s<1$ (see \cite{FMM}
and \cite{JK}). Hence, $\Phi_k \ri 0 $ in $H^{-s +\frac12} (\pa
\Om)$. But, it contradicts to the fact that $\|\Phi_k\|_{H^{-s + \frac12}(\pa
\Om)} =1$ for each $k$ and this again implies that $\{S_s\phi_k\}$ cannot be unbounded.
Therefore, we conclude that $S_s : H^{-s + \frac12} (\pa \Om)
\ri H^{s -\frac12} (\pa \Om)$ has a closed range.
\end{proof}

To show that $S_s : H^{-s+ \frac12}(\pa \Om) \ri H^{s-\frac12}(\pa
\Om)$ is bijective, it remains to show  $S_s$ has a dense range.
For the purpose of it, we show that dual operator $S_s^*: H^{-s +
\frac12}(\pa \Om) \ri (H^{-s + \frac12} (\pa \Om))^*$ of $S_s$ is
one-to-one.   Suppose that $S_s^* \phi =0 $. Then, we have
\begin{align*}
0=\ll S_s^* \phi, \psi\gg = <\phi, S_s \psi>
\end{align*}
for all $\psi \in H^{-s + \frac12}(\pa \Om)$, where $\ll\cdot, \cdot\gg$
is the duality paring between $H^{-s + \frac12}(\pa \Om)$ and $ (H^{-s
+ \frac12} (\pa \Om))^*$. Let $u = {\mathcal S}_s \phi$. By
\eqref{paring}, taking $\psi = \phi$, we have
\begin{align*}
0= < \phi, S_s \phi> = \int_{\R} |\xi|^{2s}|\hat{ u}(\xi)|^2 d\xi.
\end{align*}
This implies that $u=0$.
By Lemma \ref{lemma1}, we conclude that  $\phi =0$. Hence
$S_s^*$ is one-to-one.  This completes the proof of the invertibility
of $S_s$.

\section{Proof of Theorem \ref{theo2-1}}
\setcounter{equation}{0}

Let $f \in H^{s-\frac12}(\pa \Om)$. By Theorem \ref{theo1}, there is a $\phi \in H^{-s +\frac12}(\pa \Om)$ such that $S_s \phi = f$.
Let $u(x) = {\mathcal S}_s \phi (x)$ for $x \in \R$. Clearly, $u|_{\pa \Om} = S_s \phi =f$ on $\pa \Om$.
Moreover, by theorem \ref{rem2} and \eqref{solution}, we have that $\De^s u =0$ in $\R \setminus \pa \Om$ and $u \in \dot{H}^s(\R)$.
Hence, we showed the existence of the solution of equation \eqref{main result}. For the uniqueness, assume that $u$ is a solution of \eqref{main result}
such that $u|_{\pa \Om} =0$. Since $\De^s u \in \dot{H}^{-s}(\R) =(\dot{H}^{s}(\R))^*$ and $u|_{\pa \Om} =0$, we have
\begin{align*}
0= <\De^s u, u>  = -\int_{\R} (2\pi i|\xi|)^{2s}|\hat u(\xi)|^2 d\xi,
\end{align*}
where $< \cdot, \cdot>$ is the duality paring between $\dot{H}^s(\R)$ and $\dot{H}^{-s}(\R)$.
This implies that $\hat u =0$ and hence, $u$ is constant in $\R$. Since $u|_{\pa \Om} =0$, we have that $u \equiv 0$.
Hence, the solution of \eqref{main result} is unique. $\Box$

\end{document}